\documentclass[11pt,a4paper]{article}
\usepackage{amsmath,amsthm}
\usepackage{amsfonts,amstext,amssymb, mathtools, comment}

\newcommand{\p}{{\mathbb P}}
\newcommand{\e}{{\mathbb E}}
\renewcommand{\a}{\alpha}
\newcommand{\D}{\mathrm d}
\newcommand{\levy}{L\'{e}vy }

\newcommand{\R}{\mathbb R}
\renewcommand{\Re}{{\rm Re}}

\newcommand{\1}[1]{\mbox{\large  1}_{\{#1\}}}

\newcommand{\uX}{\underline X}
\newtheorem{thm}{Theorem}
\newtheorem{lem}[thm]{Lemma}

\newtheorem{cor}[thm]{Corollary}

\theoremstyle{remark}

\title{A new approach to fluctuations of a reflected \levy process}

\begin{document}
\bibliographystyle{plain}
\title{A new approach to fluctuations of reflected L\'{e}vy
processes}\author{Jevgenijs Ivanovs\footnote{Eurandom, Eindhoven
University of Technology and Korteweg-de Vries Institute for
Mathematics, University of Amsterdam}} \maketitle \abstract{We
present a new approach to fluctuation identities for reflected
\levy processes with one-sided jumps. This approach is based on a
number of easy to understand observations and does not involve
excursion theory or It\^{o} calculus. It also leads to more
general results.}

\vspace{0.1in}\noindent

\section{Introduction}
Let $X(t),t\geq 0$ be a \levy process with no positive jumps. We
consider $X(t)$ and $-X(t)$ reflected at 0, and study the times at
which these reflected processes pass over a certain level $B>0$.
In fact, for each reflected process we are interested in the joint
Laplace transform of the first passage time, the overshoot, and
the corresponding value of the local time at 0. In addition, it is
assumed that the initial value of $X(t)$ is shifted to an
arbitrary $x_0\in[0,B]$.

The main idea is to add an additional reflecting barrier at $B$.
That is, we consider a process $W(t),t\geq 0$ with values in
$[0,B]$ having the representation
\[W(t)=X(t)+L(t)-U(t),\]
where $L(t)$ and $U(t)$ are non-decreasing right-continuous
functions called the \emph{local times} at respectively the lower
and the upper barriers (that is at 0 and at $B$). In addition, it
is required that $L(0)=U(0)=0$ and the points of increase of
$L(t)$ and $U(t)$ are contained in the sets $\{t\geq 0:W(t)=0\}$
and $\{t\geq 0:W(t)=B\}$ respectively. It is known that the
triplet of functions $(W(t),L(t),U(t))$ exists and is unique, see
e.g.~\cite{ramanan_skorohod0a}, and is called the solution of the
two-sided Skorokhod problem. Letting $B=\infty$ we obtain a
one-sided reflection at 0, in which case $U(t)=0$ and $L(t)$ can
be given explicitly through $L(t)=-\min(\underline X(t),0)$, where
$\underline X(t)=\inf\{X(s):0\leq s\leq t\}.$

Define the inverse local times through
\begin{align}&\tau^L_x=\inf\{t\geq 0:L(t)>x\},&\tau^U_x=\inf\{t\geq 0:U(t)>x\}.\end{align}
Note that the one-sided reflection of $X(t)$ at 0 behaves as
$W(t)$ up to time $\tau^U_0$, its first passage time over level
$B$ is given by $\tau^U_0$, and the corresponding value of the
local time at 0 is $L(\tau^U_0)$. Absence of positive jumps
implies that $U(\tau_0^U)=0$, in other words, there is no
overshoot. One of our goals is to characterize the distribution of
$(\tau^U_0,L(\tau_0^U))$ for an arbitrary starting point
$x_0\in[0,B]$. Similarly, looking down from the level $B$ we see
that $(\tau_0^L,L(\tau_0^L),U(\tau_0^L))$ describes the first
passage time, the overshoot, and the local time at 0 of the
one-sided reflection of $-X(t)$ with the starting point $B-x_0$.
Our second goal is to characterize this triplet.

In order to achieve the above goals we study the process
$L(\tau_x^U),x\geq 0$, which turns out to be a compound Poisson
process with some specific distribution of $L(\tau_0^U)$. We
determine this initial distribution, the distribution of jumps and
the jump arrival rate. Based on these key results we derive the
Laplace transforms of $(\tau^U_0,L(\tau_0^U))$ and
$(\tau_0^L,L(\tau_0^L),U(\tau_0^L))$ for an arbitrary
$x_0\in[0,B]$. To our knowledge, these results are more general
than the existing ones. We advise the reader to have a look
at~\cite[Ch.\ 8.5]{kyprianou}, where the identities for reflected
\levy processes are presented and the relevant literature is
discussed. The main references
are~\cite{avram,pistorius,korolyuk}. The proofs in the cited
papers are rather involved and are based on It\^{o}'s excursion
theory, stochastic calculus and martingale calculations. A proof
based on direct excursion theory calculations is given
in~\cite{doney}. Our approach employs a number of easy to understand
observations and does not require the use of above mentioned heavy
machinery.

The following section contains some preliminary material. The main
results under a tilted measure are derived in
Section~\ref{sec:main}. Finally, Section~\ref{sec:last} contains
our results in their general form.

\section{Preliminaries}\label{sec:prelim}
Let $X(t),t\geq 0$ be a \levy process defined on $(\Omega,\mathcal
F,\{\mathcal F_t\}_{t\geq 0},\p^0)$. It is assumed that $X(t)$ has
no positive jumps, moreover, it is not a process with
a.s.~monotone paths. Then $\e^0 e^{\a X(t)}=e^{\phi(\a)t}$ for
some function $\phi(\a)$, called Laplace exponent, and
$\Re(\a)\geq 0$. It is useful to note that $\phi(\a)$ is analytic
in the domain specified by $\Re(\a)>0$. Restricting ourselves to
the real half line $\a \geq 0$, we note that $\phi(\a)$ is convex,
$\phi(\a)\rightarrow\infty$ as $\a\rightarrow\infty$, and $\e
X(1)=\phi_+'(0)$, where $\phi_+'$ denotes the right derivative of
$\phi$.

Pick $q>0$ and denote the unique positive solution of $\phi(\a)=q$
through $\Phi(q)$, which is sometimes called the right inverse of
$\phi(\a)$. Consider the Wald's martingale
$e^{\Phi(q)X(t)-qt},t\geq 0$ and the corresponding measure $\p$,
that is, a measure with Radon-Nikodym derivative
\begin{equation}\label{eq:RN}\frac{\D\p}{\D\p^0}|_{\mathcal F_t}=e^{\Phi(q)X(t)-qt}.\end{equation} It
is easy to see that $X(t)$ is a \levy process under $\p$ and its
Laplace exponent is
\begin{equation}\label{eq:exponent}\psi(\a)=\phi(\a+\Phi(q))-q.\end{equation} Hence $\psi(\a)$
is analytic in the neighborhood of 0, and
$\psi'(0)={\phi}'(\Phi(q))>0$ implying that~$X(t)$ has a positive
drift under~$\p$.

Let us proceed to so-called $q$-scale functions, which play a
central role in the theory of fluctuations of one-sided \levy
processes. It is known that there exists a unique strictly
increasing and continuous function $W^{(q)}(x),x\geq 0$ satisfying
$\int_0^\infty e^{-\a x}W^{(q)}(x)\D x=1/(\phi(\a)-q)$ for
$\a>\Phi(q)$. Such a function can be given explicitly through
\begin{align}\label{eq:scale_function}&W^{(q)}(x)=e^{\Phi(q)x}W(x), &W(x)=\frac{1}{\psi'(0)}\p(\uX\geq
-x),\end{align} where $\uX$ denotes the all-time infimum of
$X(t)$, see~\cite[Ch.~8.2]{kyprianou}. In addition, $W^{(0)}(x)$
is defined as a limit of $W^{(q)}(x)$ as $q\downarrow 0$.

Define the first passage times $\tau^\pm_x=\inf\{t\geq 0:\pm
X(t)>x\}$ for $x\geq 0$ and note using the strong Markov property
that
\begin{align*}&\p(\uX\geq -a)=\p(\tau_b^+<\tau^-_a)\p(\uX\geq
-(a+b)),& a,b\geq 0.\end{align*} Hence
$\p(\tau_b^+<\tau_a^-)=W(a)/W(a+b)$ for $a+b>0$. In terms of the
original measure $\p(\tau_b^+<\tau_a^-)$ reads as $
\e^0[e^{\Phi(q)X(\tau_b^+)-q\tau_b^+};\tau_b^+<\tau_a^-]$ and
hence
\[\e^0(e^{-q\tau_b^+};\tau_b^+<\tau_a^-)=\frac{W^{(q)}(a)}{W^{(q)}(a+b)}.\]
Finally, we note that $W(x)$ and hence also $W^{(q)}(x)$ have right derivatives for $x>0$.

\section{Results under the tilted measure $\p$}\label{sec:main}
The main object of our study is the process $L(\tau_x^U),x\geq 0$,
which has non-decreasing paths. Using the strong Markov property
we see that $L(\tau_x^U)$ is a \levy process. In fact, it is a
compound Poisson process, because it does not jump in a fixed
interval with positive probability. The following theorem
identifies this process.
\begin{thm}\label{thm:main}
The process $L(\tau_x^U),x\geq 0$ is a compound Poisson process
characterized by
\begin{align}
\psi^L(\a)=\frac{1}{x}\log \e e^{-\a [L(\tau_x^U)-L(\tau_0^U)]}=\frac{W(B)\psi(\a)}{Z(\a,B)}-\a,\\
\e_{x_0}e^{-\a
L(\tau^U_0)}=\frac{Z(\a,x_0)}{Z(\a,B)},\label{eq:m2}
\end{align}
where $\Re(\a)\geq 0$ and $Z(\a,x)=e^{\a
x}\left(1-\psi(\a)\int_0^x e^{-\a y}W(y)\D y\right)$.
\end{thm}
This result is based on the following lemma, which is closely
related to the proof of~\cite[Theorem 4.1]{MMBM_2sided}, where a
Markov-modulated Brownian motion is considered. Similar ideas in a
simpler form also appear in~\cite[Section 5]{rogers}.
\begin{lem}\label{lem:tech}
It holds for $\a<0$ that \begin{align*}&\int_0^\infty\e_{x_0}
e^{\a X(\tau_x^U)}\D x\\&=\left(-e^{\a
B}/\a+e^{\a(B+x_0)}\int_{x_0}^\infty e^{-\a y}\p(\uX\leq -y)\D
y\right)/(1-\p(\uX\leq -B)).\end{align*}
\end{lem}
It is noted that the expression in the lemma might not be finite.
\begin{proof}
Pick an arbitrary $y\in\R$ and consider the time points $t\geq 0$
such that $X(t)=y$ and $U(t-)<U(t+s),\forall s>0$ ($t$ is the
point of increase of $U$). We denote this set through $T_y$. If
$y\geq B$ then the first such point is $\tau^+_y$; if $y<B$ then
it is $\inf\{t\geq 0:\uX(t)\leq y-B,X(t)>y\}$. To justify the
second statement, note that $X(t)$ should first hit level $y-B$ to
guarantee that $W(t)=B$ at the time when $X(t)$ hits $y$; drawing
a picture might be helpful here. Recall that $X(t)$ has a positive
drift under $\p$, hence $T_y$ contains at least one point with
probability $1$ or $\p_{x_0}(\uX\leq y-B)$ corresponding to $y\geq
B$ and $y<B$. Use the strong Markov property to see that there are
at least $n\geq 1$ points in $T_y$ with probability $\p(\uX\leq
-B)^{n-1}$ or $\p_{x_0}(\uX\leq y-B)\p(\uX\leq -B)^{n-1}$
depending on $y$. Hence the expected number of points in $T_y$ is
\[\e_{x_0}|T_y|=\begin{cases}(1-\p(\uX\leq -B))^{-1},&y\geq B\\
 \p(\uX\leq y-B-x_0)(1-\p(\uX\leq -B))^{-1},&y<B.\end{cases}\]
Multiply both sides by $e^{\a y}$, where $\a<0$, and integrate
over $y$ to see that the right side is as stated in the lemma. It
therefore remains to show that
\[\int_{-\infty}^{\infty}e^{\a y}\e_{x_0}|T_y|\D
y=\int_0^\infty\e_{x_0} e^{\a X(\tau_x^U)}\D x.\]
 But,
$|T_y|=\sum_{x\geq 0}\1{X(\tau_x^U)=y}$ hence, in view of Fubini's
theorem, it is sufficient to show that
\[\int_{-\infty}^{\infty}\sum_{x\geq 0}e^{\a y}\1{X(\tau_x^U)=y}\D
y=\int_0^\infty e^{\a X(\tau_x^U)}\D x\] holds $\p$-a.s. For this
note that $X(\tau_x^U)=-L(\tau_x^U)+x+B$ and $L(\tau_x^U)$ is
piecewise constant. Let $L(\tau_x^U)=C$ on the interval $[S,F)$
then
\[\int_{-\infty}^{\infty}\sum_{x\in[S,F)}e^{\a y}\1{X(\tau_x^U)=y}\D
y=\int_S^F e^{\a X(\tau_x^U)}\D x.\] Summing over all such
intervals concludes the proof.
\end{proof}

\begin{proof}[Proof of Theorem~\ref{thm:main}]
We start with generalized Pollaczek-Khinchine formula:
\begin{align}\label{eq:PK}&\e e^{\a\uX}=\psi'(0)\a/\psi(\a),&\a>0,\end{align}
see also~\cite[p.\ 217]{kyprianou}. Let us show that this identity
can be extended to some negative values of $\a$. Let
\[r=\inf\{\a\in\R:\e e^{\a \uX}<\infty\},\] which is non-positive.
It is well known that $\e e^{\a\uX}$ is analytic for $\Re(\a)>r$.
Note that $\e e^{\a X(1)}$ is finite for $\a>r$, and hence
$\psi(\a)$ is analytic in the same domain. Therefore $\e
e^{\a\uX}\psi(\a)-\psi'(0)\a$ is identically 0 here, and
so~(\ref{eq:PK}) holds true for $\Re(\a)>r,\a\neq 0$.
Let us show that $r<0$. If $r=0$ then the right hand side of~(\ref{eq:PK}) should have a singularity at 0,
see e.g.~\cite[Ch.\ II, Thm.\ 5b]{widder}. This is not the case, because
$\psi(\a)$ is analytic in the neighborhood of 0 and $\psi'(0)>0$ according to representation~(\ref{eq:exponent}).

 Pick an arbitrary $\a\in(r,0)$ and
apply Fubini's theorem to see that
\begin{align*}\int_0^\infty e^{-\a y}\p(-\uX\geq y)\D
y=\int_{0-}^\infty\int_{0}^xe^{-\a y}\D y\p(-\uX\in \D
x)=\frac{1}{\a}(1-\e e^{\a\uX}),\end{align*} which is further
written as $1/\a-\psi'(0)/\psi(\a)$.
 Note also that
$\p(\uX\leq -y)=\p(\uX<-y)$ for $y>0$, because $X(t)$ can hit $-y$
only if it has paths of unbounded variation~\cite[Theorem
7.11]{kyprianou}, but then it will pass over $-y$ a.s. Then by the
definition of $W(y)$ we have $\p(\uX\leq -y)=1-\psi'(0)W(y)$ for
$y>0$, and thus
\[\int_0^{x_0} e^{-\a y}\p(\uX\leq -y)\D
y=\frac{1}{\a}(1-e^{-\a x_0})-\psi'(0)\int_0^{x_0}e^{-\a y}W(y)\D
y.\] This yields an expression for $\int_{x_0}^\infty e^{-\a
y}\p(\uX\leq -y)\D y$, which together with Lemma~\ref{lem:tech}
imply
\[\int_0^\infty\e_{x_0} e^{\a
X(\tau_x^U)}\D x=e^{\a(B+x_0)}\left(\int_0^{x_0}e^{-\a y}W(y)\D y
-\frac{1}{\psi(\a)}\right)/W(B).\]

Noting that $X(\tau_x^U)=-L(\tau_x^U)+x+B$ we obtain
\[\int_0^\infty\e_{x_0} e^{\a
X(\tau_x^U)}\D x=\e_{x_0}e^{-\a L(\tau_0^U)}e^{\a B}\int_0^\infty
e^{\a x+\psi^L(\a)x} \D x,\] which is convergent as shown above.
Therefore
\begin{equation*}\frac{\e_{x_0}e^{-\a
L(\tau_0^U)}}{\psi^L(\a)+\a}=\frac{Z(\a,x_0)}{W(B)\psi(\a)}
\end{equation*} for all $\a\in (r,0)$. The expressions in the statement of the theorem are
obtained by noting that $\tau_0^U=0$ a.s.\ given $x_0=B$. It
remains to show that these identities hold true for $\Re(\a)>r$.
Multiplied by $Z(\a,B)$ they indeed hold true, because the Laplace transform $\int_0^x e^{-\a
y}W(y)\D y$ is entire, see e.g.~\cite[Ch.\ II, Lem.\ 5]{widder}, and hence $Z(\a,x)$ is analytic in
$\Re(\a)>r$. Finally, if $Z(\a,B)=0$ for some $\a$ in the domain of interest then $\psi(\a)=0$ implying $Z(\a,B)=e^{\a B}\neq 0$, a contradiction.
\end{proof}

Theorem~\ref{thm:main} specifies the compound Poisson process
$L(\tau_x^U),x\geq 0$ uniquely, hence it can potentially be used
to obtain the jump arrival rate of this process. It is, however,
easier to give a direct argument. Note that a jump of
$L(\tau_x^U)$ corresponds to an excursion from the maximum of
height exceeding $B$. Such excursions arrive with rate
$W'_+(B)/W(B)$, see e.g.~\cite[p. 220]{kyprianou}, thus we have
the following result.
\begin{lem}\label{lem:rate}
The jump arrival rate of $L(\tau_x^U)$ is
$\lambda_L=W'_+(B)/W(B)$.
\end{lem}
Let us sketch an alternative proof of the above lemma without an explicit use of excursion theory.
First, note that the probability that $L(\tau_x^U)$ jumps in $[0,\epsilon]$ is bounded from below by
$\p(\tau_B^-<\tau_\epsilon^+)=1-W(B)/W(B+\epsilon)$. This shows that $W'_+(B)/W(B)$ is the lower bound on the jump arrival rate. This lower bound is achieved if $\p(\tau_{B-\epsilon}^-<\tau_\epsilon^+)\p(\tau_B^+<\tau_\epsilon^-)/\epsilon$ goes to 0 as $\epsilon\downarrow 0$,
because of the strong Markov property. The alternative proof is completed after one observes that $\p(\tau_{B-\epsilon}^-<\tau_\epsilon^+)/\epsilon$ is bounded for small enough $\epsilon$.

Let us now present an important corollary of
Theorem~\ref{thm:main}.
\begin{cor}\label{cor:main}
It holds for $\a,\theta\geq 0$ that
\[\e_{x_0} e^{-\a L(\tau_0^L)-\theta U(\tau_0^L)}=Z(\a,x_0)+\frac{W(x_0)[W(B)\psi(\a)-(\a+\theta)Z(\a,B)]}{W'_+(B)+\theta W(B)}.\]
\end{cor}
\begin{proof}
Let $\Delta=\inf\{x\geq 0:L(\tau_x^U)>0\}$ then according to
Lemma~\ref{lem:Poisson} we have
\[\e_B e^{-\a L(\tau_\Delta^U)-\theta \Delta}=\frac{\lambda_L+\psi^L(\a)}{\lambda_L+\theta},\]
where $\a,\theta\geq 0$. Next we write
\begin{align*}
\e_{x_0}e^{-\a L(\tau_\Delta^U)-\theta\Delta}=\e_{x_0}[e^{-\a
L(\tau_0^U)};\tau_0^-<\tau_B^+]+\p_{x_0}(\tau_B^+<\tau_0^-)\e_B e^{-\a L(\tau_\Delta^U)-\theta \Delta}\\
=\e_{x_0}e^{-\a
L(\tau_0^U)}-\p_{x_0}(\tau_B^+<\tau_0^-)+\p_{x_0}(\tau_B^+<\tau_0^-)\e_B
e^{-\a L(\tau_\Delta^U)-\theta \Delta}.
\end{align*}
But the strong Markov property implies that
\[\e_{x_0} e^{-\a L(\tau_\Delta^U)-\theta \Delta}=\e_{x_0} e^{-\a L(\tau_0^L)-\theta U(\tau_0^L)}\e_0 e^{-\a L(\tau_0^U)}.\]
Hence
\[\e_{x_0} e^{-\a L(\tau_0^L)-\theta U(\tau_0^L)}=\left(\e_{x_0}e^{-\a
L(\tau_0^U)}+\p(\tau_{B-x_0}^+<\tau^-_{x_0})\frac{\psi^L(\a)-\theta}{\lambda_L+\theta}\right)/\e_0
e^{-\a L(\tau_0^U)}.\] Recall that
$\p(\tau_{B-x_0}^+<\tau^-_{x_0})=W(x_0)/W(B)$, and apply
Lemma~\ref{lem:rate} and Theorem~\ref{thm:main} to conclude the
proof.
\end{proof}

\section{Back to the original measure}\label{sec:last}
In this section we rewrite the results obtained in the previous
section in terms of the original measure $\p^0$.
\begin{thm}
It holds for $q>0,\a\geq \Phi(q),\theta\geq 0$ that
\begin{align*}
\e^0_{x_0}&e^{-\a L(\tau_0^L)-\theta
U(\tau_0^L)-q\tau_0^L}\\&=Z^{(q)}(\a,x_0)+\frac{W^{(q)}(x_0)[W^{(q)}(B)(\phi(\a)-q)-(\a+\theta)Z^{(q)}(\a,B)]}{{W^{(q)}}'_+(B)+\theta
W^{(q)}(B)},\\
\e^0_{x_0}& e^{-\a
L(\tau_0^U)-q\tau_0^U}=\frac{Z^{(q)}(\a,x_0)}{Z^{(q)}(\a,B)},
\end{align*}
where \[Z^{(q)}(\a,x)=e^{\a x}\left(1+(q-\phi(\a))\int_0^x e^{-\a
y}W^{(q)}(y)\D y\right).\] Moreover, the two-dimensional \levy
process $(L(\tau_x^U),\tau_x^U),x\geq 0$ is characterized by
\[\log\e^0 e^{-\a [L(\tau_1^U)-L(\tau_0^U)]-q[\tau_1^U-\tau_0^U]}=\frac{W^{(q)}(B)(\phi(\a)-q)}{Z^{(q)}(\a,B)}-\a\]
with $q> 0,\a\geq \Phi(q)$.
\end{thm}
\begin{proof}
We only prove the first identity and note that the two other ones
follow in a similar way from Theorem~\ref{thm:main}. Consider
Corollary~\ref{cor:main} with $\a',\theta'\geq 0$ and write the
left hand side of the corresponding identity in terms of the
original measure $\p^0$:
\[\e^0_{x_0} e^{-\a' L(\tau_0^L)-\theta' U(\tau_0^L)}e^{\Phi(q)[X(\tau_0^L)-x_0]-q\tau_0^L}=e^{-\Phi(q)x_0}\e^0_{x_0}e^{-\a L(\tau_0^L)-\theta U(\tau_0^L)-q\tau_0^L},\]
where $\a=\a'+\Phi(q)$ and $\theta=\theta'-\Phi(q)$. Moreover, it
is easy to see that $e^{\Phi(q)x}Z(\a',x)=Z^{(q)}(\a,x)$ and
${W^{(q)}}'_+(B)=\Phi(q)W^{(q)}(B)+e^{\Phi(q)B}W'_+(B)$. Simple
algebraic manipulations complete the proof.
\end{proof}

We remark that the above identities also hold for $q=0$. Let
$q\downarrow 0$ to see this. The only technical difficulty arises
when considering the first identity. Namely, one has to show that
$\lim_{q\downarrow 0}{W^{(q)}}'_+(B)={W^{(0)}}'_+(B).$ This
difficulty is bypassed by noting that Lemma~\ref{lem:rate} and
then also Corollary~\ref{cor:main} hold under the original
measure.

\appendix
\section*{Appendix}
\begin{lem}\label{lem:Poisson}
Let $X(t)$ be a compound Poisson process with rate $\lambda$ and
Laplace exponent $\log\e e^{-\a X(1)}=\psi(\a)$. Let $J$ be the
time of the first jump of $X(t)$ then
\[\e e^{-\a X(J)-\theta J}=\frac{\lambda+\psi(\a)}{\lambda+\theta},\]
where $\a,\theta\geq 0$.
\end{lem}
\begin{proof}
Let $\eta$ be an exponential random variable of rate $\theta>0$
independent of everything else. Clearly, $\e e^{-\a
X(\eta)}=\theta/(\theta-\psi(\a))$ and $\p(J>\eta)=\e e^{-\lambda
\eta}=\theta/(\theta+\lambda)$. But we also have
\[\e e^{-\a X(\eta)}=\p(J>\eta)+\e [e^{-\a X(J)};J<\eta]\e[e^{-\a X(\eta)}],\]
which results in the statement of the Lemma for $\theta>0$. Taking
limits as $\theta\downarrow 0$ proves the case of $\theta=0$.
\end{proof}

\bibliography{levy_fluctuations}
\end{document}